\documentclass{amsart}

\usepackage{amssymb}
\usepackage{graphicx}
\usepackage{MnSymbol}

\usepackage{enumerate}

\usepackage{amsmath,amscd}

\usepackage{amsthm}

\title{Homomorphisms from topological groups to inverse limits}

\author{Gregory R. Conner}
\address{Brigham Young University, Department of Mathematics, Provo, UT 84602, USA}
\email{conner@math.byu.edu}

\author{Samuel M. Corson}
\address{E. T. S. I. I. Universidad Polit\'{e}cnica de Madrid, Jos\'{e} Guti\'{e}rrez Abascal 2, 28006 Madrid, Spain}
\email{sammyc973@gmail.com}

\author{Curtis Kent}
\address{Brigham Young University, Department of Mathematics, Provo, UT 84602, USA}
\email{curtkent@mathematics.byu.edu}

\theoremstyle{definition}\newtheorem{theorem}{Theorem}
\theoremstyle{definition}
\theoremstyle{definition}

\theoremstyle{definition}

\theoremstyle{definition}

\theoremstyle{definition}
\theoremstyle{definition}

\theoremstyle{definition}
\theoremstyle{definition}\newtheorem{corollary}[theorem]{Corollary}
\theoremstyle{definition}
\theoremstyle{definition}\newtheorem{definition}[theorem]{Definition}
\theoremstyle{definition}
\theoremstyle{definition}
\theoremstyle{definition}\newtheorem{remark}[theorem]{Remark}
\theoremstyle{definition}
\theoremstyle{definition}\newtheorem{lemma}[theorem]{Lemma}
\theoremstyle{definition}
\theoremstyle{definition}
\theoremstyle{definition}
\theoremstyle{definition}
\theoremstyle{definition}
\theoremstyle{definition}
\theoremstyle{definition}

\newcommand{\Aut}{\operatorname{Aut}}

\newcommand{\Nb}{\operatorname{Nb}}
\newcommand{\Orb}{\operatorname{Orb}}
\newcommand{\Z}{\mathbb{Z}}
\newcommand{\N}{\mathbb{N}}

\usepackage{color}

\begin{document}



\thanks{The second author is supported by RYC2023-045493-I.  The third  author is supported by Simons collaboration grant 587001.}

\begin{abstract}
We prove a general theorem giving constraints on maps from certain topological groups to inverse limits of bounded torsion groups.  From this we obtain some automatic continuity and ultraproduct results.  For example, every homomorphism from a Polish group to a countable torsion-free residually finite group has open kernel.  Also, the Grigorchuk group is a homomorphic image of a nonprincipal ultraproduct of groups if and only if there exists a measurable cardinal.
\end{abstract}

\maketitle

\begin{section}{Introduction}

Abstract homomorphisms from large topological groups into algebraically restricted groups often display strong rigidity phenomena.   For example, every homomorphism from $\Z^\N$ to a free abelian group factors through a projection to a finite sub-product $\Z^n$. Abelian groups with this property are called \emph{slender}.  Nunke in \cite {Nunke} gave a complete classification of the class of abelian slender groups. Additional generalizations of slenderness were introduced by Eda in 1992 \cite{Eda1992} and Corson and Conner in 2019 \cite{CoCo}.

%G\"obel demonstrated that any (possibly non-abelian) group $H$ is slender (every homomorphism from $\Z^\N$ to $H$ factors through a projection to a finite sub-product)  if and only if all abelian subgroups of $H$ are slender \cite{Gobel}.  As an approach to defining slenderness for non-abelian groups this is unsatisfactory, since homomorphisms from abelian groups to non-abelian groups are always highly restricted. 

Eda defined a group $H$ to be \emph{n-slender} if  every homomorphism from the fundamental group of the infinite earring to $H$ factors through a finite rank free group \cite{Eda1992}.  This latter condition is equivalent to the kernel being open when the fundamental group is endowed with the shape topology.  Corson and Conner define a group $H$ to be \emph{cm-slender} (respectively \emph{lcH-slender}) if every homomorphism from a completely metrizable (resp. locally compact Hausdorff) topological group to $H$ has open kernel.  Both of these generalizations have received significant attention \cite{BogopolskiCorson, Co2, KMV, KramerVarghese2019, MollerParisVarghese2024, MollerVarghese} and are included in the broad genre of automatic continuity, since any abstract homomorphism from, say, a completely metrizable topological group to a cm-slender group $H$ will be continuous (when $H$ is endowed with any topology).

%A subgroup of a slender (in any of these senses) group is again slender, and typical examples of non-slender groups include nontrivial finite groups and the additive group $\mathbb{Q}$.  

A complete classification of lcH-slender groups, in terms of subgroups,  was obtained by Corson and Varghese \cite{CoVarghese}. Thus one is lead to consider more generally groups which are locally \emph{countably} compact Hausdorff slender (abbreviated \emph{lccH-slender} and defined in the natural way \cite{BogopolskiCorson}).

In the current paper we prove a rather general, technical result (Theorem \ref{main}) which states that under appropriate conditions a homomorphic image of a topological group $\mathcal{G}$ has a large abelian subgroup. The groups $\mathcal{G}$ that we consider are those already mentioned (completely metrizable, etc.) and the codomain of the homomorphism will be an inverse limit of so-called \emph{bounded torsion groups} (i.e. groups satisfying some Burnside law $x^N = 1$).

One easy-to-state consequence is the following, which is a simplification of the conjunction of Corollary  \ref{torsionfree cm-slender} with Corollary \ref{torsionfree n-slender}.

\begin{theorem}\label{introductionthm1}
    Every torsion-free residually finite group of size less than $2^{\aleph_0}$ is n-slender, cm-slender, and lccH-slender.
\end{theorem}

For example, the Basilica group \cite{GrigorchukZuk} satisfies the hypotheses of Theorem \ref{introductionthm1}.  Note that the assumption ``torsion-free'' cannot be dropped, since nontrivial torsion implies that a group is not slender.  Also, the assumption regarding the size cannot be dropped since the group $\Z^\N$ is completely metrizable, residually finite and the identity endomorphism witnesses that the group is not cm-slender. The group $\Z^\N$ is also a homomorphic image, via a homomorphism with non-open kernel, of the infinite earring group and so is not n-slender.

A consequence of a celebrated result of Nikolov and Segal states that if $\phi$ is an abstract homomorphism from a topologically finitely generated profinite group $\mathcal{G}$ to a residually finite group then $\ker(\phi)$ is closed (see \cite[p. 519 (F)]{NS2}).  By contrast, Theorem \ref{introductionthm1} strengthens the requirements on the codomain, loosens the requirements on the domain, and strengthens the conclusion regarding the kernel.

Other consequences of Theorem \ref{main} allow the presence of torsion in the codomain.  In the terminology of \cite{Co}, a homomorphic image of a nonprincipal ultraproduct of groups $\prod_{\N} G_n/\mathcal{U}$ is a \emph{ui} group.  Compact groups (particularly finite groups) and vector spaces over a field are ui groups.  Groups having nontrivial cm-slender homomorphic quotients are not ui.  Some examples of (infinitely generated) torsion groups which are not ui were given in \cite{Co}, but no finitely generated torsion examples were produced.

A second consequence of Theorem \ref{main} states that a small residually (bounded torsion) group having abelian subgroups of unbounded torsion is not u.i. (see Theorem \ref{withultraprod}).  This can be applied, for example, to weakly branch subgroups of the automorphism group of a finitely branching rooted tree.  As a result, Grigorchuk's group (which is finitely generated torsion) is not ui.  Combined with some new characterizations of ui groups developed in Section \ref{sec: applications to ui}, this gives rise to the rather eclectic assertion announced in the abstract (see Corollary \ref{Ggroupmeasurable}).

\begin{theorem} Grigorchuk's group is a homomorphic image of a nonprincipal ultraproduct of groups if and only if there exists a measurable cardinal.
\end{theorem}

The paper is organized as follows. In Section \ref{sec: main result}, we prove our structure theorem for homomorphisms into inverse limits of bounded torsion groups and derive its consequences for residually (bounded torsion) groups. Section \ref{sec: applications to ui} applies these ideas to ultraproduct images and measurable cardinals. Section \ref{sec: applications to rooted trees} treats automorphism groups of rooted trees and weakly branch groups, culminating in the application to Grigorchuk's group.

\end{section}

\begin{section}{Main theorem}\label{sec: main result}

Our notation will be standard.  We let $|X|$ denote the cardinality of a set $X$, $2^{\aleph_0}$ be the cardinality of the continuum, and $f \upharpoonright Y$ be the restriction of the function $f$ to a subset $Y$ of the domain.  We shall consider inverse systems $\pi_{i+1, i}: H_{i + 1} \rightarrow H_i$ of group homomorphisms indexed by the set $\mathbb{N}$ of natural numbers.  We do not require the $\pi_{i+1, i}$ to be surjective.  In this context, let $\pi_j: \varprojlim H_i \rightarrow H_j$ denote the natural homomorphic projection.  Let $O(g)$ denote the order of a group element $g$, so that $O(\cdot)$ takes values in $(\mathbb{N} \setminus \{0\}) \cup \{\infty\}$.  For a topological group $\mathcal{G}$, let $\Nb(\mathcal{G}, 1)$ denote the collection of open neighborhoods of identity in $\mathcal G$.

\begin{theorem}\label{main}  Suppose that

\begin{itemize}

\item $G$ is a finite index subgroup of a topological group $\mathcal{G}$, with $\mathcal{G}$ completely metrizable or locally countably compact Hausdorff;

\item $\{\pi_{i + 1, i}: H_{i + 1} \rightarrow H_i \mid i \in \mathbb{N}\}$ is a collection of group homomorphisms;

\item $\phi: G \rightarrow \varprojlim H_i$ is an abstract group homomorphism;

\item for each $j \in \mathbb{N}$ there is a natural number $N_j \in \mathbb{N}$ such that for all $g \in G$, $O(\pi_j \circ \phi(g)) \leq N_j$;

\item there is a subset $X \subseteq G$ such that $\phi(X)$ generates an abelian group and for each $U \in \Nb(\mathcal{G}, 1)$ we have $\sup\{O(\phi(g)) \mid g \in U \cap X\} = \infty$.

\end{itemize}

\noindent Then for each $V \in \Nb(\mathcal{G}, 1)$ there is a subset $\overline{X}_V \subseteq V \cap G$ such that

\begin{enumerate}

\item[$(*)_1$]  $|\overline{X}_V| = 2^{\aleph_0}$;

\item[$(*)_2$] $\phi \upharpoonright \overline{X}_V$ is injective;

\item[$(*)_3$] $\phi(\overline{X}_V)$ generates an abelian group;

\item[$(*)_4$] $\phi(\overline{X}_V)$ contains an element of infinite order.
\end{enumerate}
\end{theorem}

\begin{proof}
Assume first that the topological group $\mathcal{G}$ is a complete metric space with the metric $d$.  As $G$ is of finite index in $\mathcal{G}$, select $K \in \mathbb{N} \setminus \{0\}$ for which $\{g^K\mid g \in \mathcal{G}\} \subseteq G$ (for example, set $K$ to be the factorial of the index $|\mathcal{G} : G|$).  Let $V \in \Nb(\mathcal{G}, 1)$ be given.  We inductively construct sequences

\begin{itemize}

    \item $V_m$ of elements in $\Nb(\mathcal{G}, 1)$;

    \item $g_m$ of elements in $X$;

    \item $j_m$ of natural numbers; and

    \item $k_m$ of natural numbers.

\end{itemize}

\noindent Let $V_0 = V$, pick $g_0 \in V_0 \cap X$ with $O(\phi(g_0)) > 1$, pick $j_0$ such that $O(\pi_{j_0} \circ \phi(g_0)) > 1$, and let $k_0 = (N_{j_0}!)K$.  Let $D = \min\{d(g_0, \mathcal{G} \setminus V_0), d(1, \mathcal{G} \setminus V_0)\}$ in case $V \neq \mathcal{G}$ and otherwise let $D = 1$.  Suppose we have selected $V_0, \ldots, V_m$ and $g_0, \ldots, g_m$ and $j_0, \ldots, j_m$ and $k_0, \ldots, k_m$.  Select $V_{m + 1} \in \Nb(\mathcal{G}, 1)$ so that for any $g \in V_{m+1}$ and sequence $\epsilon_0, \ldots, \epsilon_m$ of zeros and ones, each of the following quantities is less than $\frac{D}{3^{m + 1}}$:

\begin{itemize}

\item $d(g^{k_m}, 1)$;

\item $d\big(g_m^{\epsilon_m}(g)^{k_m}, g_m^{\epsilon_m}\big)$;

\item $d\big(g_{m - 1}^{\epsilon_{m - 1}}(g_m^{\epsilon_m}(g)^{k_m})^{k_{m - 1}}, g_{m - 1}^{\epsilon_{m - 1}}(g_m^{\epsilon_m})^{k_{m - 1}}\big)$;

$\vdots$

\item $d\big(g_0^{\epsilon_0}(g_1^{\epsilon_1}(\cdots (g_m^{\epsilon_m}(g)^{k_m})^{k_{m - 1}}  \cdots)^{k_1})^{k_0}, g_0^{\epsilon_0}(g_1^{\epsilon_1}(\cdots (g_m^{\epsilon_m})^{k_{m - 1}}  \cdots)^{k_1})^{k_0}\big)$.

\end{itemize}

\noindent Pick $g_{m + 1} \in V_{m + 1} \cap X$ with $O(\phi(g_{m + 1}^{k_mk_{m - 1}\cdots k_0})) > m + 2$, pick $j_{m + 1} > j_m$ such that $O(\pi_{j_{m + 1}} \circ \phi( g_{m+1}^{k_{m} \cdots k_1k_0})) > m + 2$, and let $k_{m + 1} = N_{j_{m + 1}}!$.

Note that for each sequence $\overline{\epsilon} = (\epsilon_t)_{t \in \mathbb{N}}$ of zeros and ones and fixed $m \in \mathbb{N}$, the sequence $g_m^{\epsilon_m}(g_{m + 1}^{\epsilon_{m + 1}}(\cdots (g_n^{\epsilon_n})^{k_{n - 1}}\cdots)^{k_{m + 1}})^{k_m}$ in parameter $n \geq m$ is Cauchy and therefore converges in $\mathcal{G}$ to, say $g_{m, \overline{\epsilon}}$.  By continuity of multiplication we have $g_{m, \overline{\epsilon}} = g_m^{\epsilon_m}(g_{m + 1, \overline{\epsilon}})^{k_m}$ for every $m \in \mathbb{N}$.  Note also that for any $\overline{\epsilon}$ we have $d(g_0^{\epsilon_0}, g_{0, \overline{\epsilon}}) < D$ and so $g_{0, \overline{\epsilon}} \in V$.  Let $\overline{X}_V$ be the set of those $g_{0, \overline{\epsilon}}$ where $\overline{\epsilon}$ is a sequence of zeros and ones.

We have already seen that $\overline{X}_V \subseteq V$.  For any given $\overline{\epsilon}$, we see that $g_{0, \overline{\epsilon}} = g_0^{\epsilon_0}(g_{1, \overline{\epsilon}})^{k_0}$, $g_0^{\epsilon_0} \in X \cup \{1\} \subseteq G$, and $(g_{1, \overline{\epsilon}})^{k_0} \in G$ (as a $K$-power in $\mathcal{G}$ by our choice of $k_0$).  Thus $\overline{X}_V \subseteq V \cap G$.  Note also that given an arbitrary $n \in \mathbb{N}$ we have

$$
\begin{array}{ll}
\pi_{j_n} \circ \phi(g_{0, \overline{\epsilon}}) & = \pi_{j_n} \circ \phi(g_{0}^{\epsilon_0}(g_1^{\epsilon_1}(\cdots g_n^{\epsilon_n}(g_{n+1, \overline{\epsilon}})^{k_n}  \cdots)^{k_1})^{k_0}) \vspace*{2mm}\\
& = \pi_{j_n} \circ \phi\big(g_{0}^{\epsilon_0}(g_1^{\epsilon_1}(\cdots (g_n^{\epsilon_n})^{k_{n - 1}}  \cdots)^{k_1})^{k_0}\big)     \vspace*{2mm}\\
& = \pi_{j_n} \circ \phi\big(g_{0}^{\epsilon_0}(g_1^{\epsilon_1})^{k_0}(g_2^{\epsilon_2})^{k_1k_0} \cdots (g_n^{\epsilon_n})^{k_{n - 1} \cdots k_0}\big)    \vspace*{2mm}\\
\end{array}
$$

\noindent since $\pi_{j_n} \circ \phi(g_{n+1, \overline{\epsilon}}^{k_n})$ is trivial (recall $k_n=N_{j_n} !$ for $n>0$) and the $\phi(g_m)$ commute with each other.  Now given distinct $\overline{\epsilon}$ and $\overline{\epsilon'}$, let $n \in \mathbb{N}$ be minimal such that $\epsilon_n \neq \epsilon_n'$.   Without loss of generality, we may assume  that $\epsilon_n = 0$ and $\epsilon_n' = 1$. Then we have $\pi_{j_n} \circ \phi (g_{0, \overline{\epsilon}}^{-1}g_{0, \overline{\epsilon'}}) = \pi_{j_n} \circ \phi(g_n^{k_{n-1} \cdots k_0}) \neq 1$.  Therefore $(*)_1$ and $(*)_2$ both hold.  For sequences $\overline{\epsilon}$ and $\overline{\epsilon'}$ and any $n \in \mathbb{N}$, we know that $\pi_{j_n} \circ \phi(g_{0, \overline{\epsilon}})$ commutes with $\pi_{j_n} \circ \phi(g_{0, \overline{\epsilon'}})$, since all $\phi(g_m)$ commute.  Therefore $\phi(g_{0, \overline{\epsilon}})$ commutes with $\phi(g_{0, \overline{\epsilon'}})$ and $(*)_3$ holds.

Now we prove $(*)_4$.  For the sake of economy, let $h_m = \phi(g_m^{k_{m-1} \cdots k_1k_0})$.  Then $h_m$ and $h_n$ commute for all $n, m \in \mathbb{N}$ and $O(h_m) \geq O(\pi_{j_m} (h_m)) \geq m + 1$.  If for some $m \in \mathbb{N}$ we have $O(h_m) = \infty$, then we may already conclude $(*)_4$ by choosing $\overline{\epsilon}$ to be $1$ only  at the index  $m$ and noting that for this $\overline{\epsilon}$, $\phi(g_{0, \overline{\epsilon}}) = h_m$.  So we may assume that $O(h_m)<\infty$ for all $m \in \mathbb{N}$.  We will pick an increasing sequence $r_0 < r_1 < r_2 < \cdots$ of natural numbers.  Let $r_0 = 0$.  Select $r_1 > r_0$ large enough that $O(\pi_{j_{r_{1} - 1}}(h_{r_0})) = O(h_{r_0})$ and $O(h_{r_1}) \geq 3O(h_{r_0})$.  Assuming that $r_0, \ldots, r_{\ell}$ have been determined, select $r_{\ell + 1} > r_{\ell}$ large enough that

\begin{itemize}

\item $O(\pi_{j_{r_{\ell + 1} - 1}}(h_{r_0} \cdots h_{r_{\ell}})) = O(h_{r_0} \cdots h_{r_{\ell}})$; and

\item $O(h_{r_{\ell + 1}}) \geq (\ell + 3)O(h_{r_0}) \cdots O(h_{r_{\ell}})$.

\end{itemize}

\noindent Since all of the $h_m$ commute, we know for a fixed $\ell$ that $O(h_{r_0} \cdots h_{r_{\ell}})$ divides $O(h_{r_0}) \cdots O(h_{r_{\ell}})$, so in particular $O(h_{r_0} \cdots h_{r_{\ell}}) \leq O(h_{r_0}) \cdots O(h_{r_{\ell}})$.  Now letting $M = O(h_{r_0} \cdots h_{r_{\ell + 1}})$ we have $(h_{r_0} \cdots h_{r_{\ell + 1}})^M = 1$, which implies that $h_{r_{\ell + 1}}^M \in \langle h_{r_0} \cdots h_{r_{\ell}} \rangle$.  Thus $O(h_{r_{\ell + 1}}^M) \leq O(h_{r_0} \cdots h_{r_{\ell}}) \leq O(h_{r_0}) \cdots O(h_{r_{\ell}})$.  Since $O(h_{r_{\ell + 1}}) \geq (\ell + 3)O(h_{r_0}) \cdots O(h_{r_{\ell}})$, we see that $M \geq \ell + 3$.  Let $\overline{\epsilon}$ be given by $\epsilon_t = 1$ if and only if $t = r_{\ell}$ for some $\ell \in \mathbb{N}$.  Now for a given $\ell \in \mathbb{N}$ we have $$O(\phi(g_{0, \overline{\epsilon}})) \geq O(\pi_{j_{r_{\ell + 2} - 1}} \circ \phi(g_{0, \overline{\epsilon}})) = O(h_{r_0} \cdots h_{r_{\ell + 1}}) \geq \ell + 3.$$  So $\phi(g_{0, \overline{\epsilon}})$ has infinite order.

In the case where $\mathcal{G}$ is locally countably compact, the proof is only slightly distinct.  Given $V \in \Nb(\mathcal{G}, 1)$ we produce sequences

\begin{itemize}

\item $V_m$ of elements in $\Nb(\mathcal{G}, 1)$;

\item $g_m$ of elements in $X$;

\item $j_m$ of natural numbers;

\item $k_m$ of natural numbers.

\end{itemize}

Let $V_0 \in \Nb(\mathcal{G}, 1)$ be such that $\overline{V_0}$ is countably compact and $\overline{V_0} \subseteq V$.  Pick $g_0 \in V_0 \cap X$ with $O(\phi(g_0)) > 1$, pick $j_0$ such that $O(\pi_{j_0} \circ \phi(g_0)) > 1$, and let $k_0 = N_{j_0}!K$  (here $K$ is as in the completely metrizable case).  Select $V_1 \in \Nb(\mathcal{G}, 1)$ with $$\overline{V_1}^{k_0} \cup g_0\overline{V_1}^{k_0}\subseteq V_0.$$  Pick $g_1 \in V_1$ with $O(\phi(g_1^{k_0})) \geq 3$ and $j_1 > j_0$ with $O(\pi_{j_1} \circ \phi(g_1^{k_0})) \geq 3$ and let $k_1 = N_{j_1}!$.

Assuming that we have made the selections $V_0, \ldots, V_m$ and $g_0, \ldots, g_m$ and $j_0, \ldots, j_m$ and $k_0, \ldots, k_m$ we select $V_{m + 1} \in \Nb(\mathcal{G}, 1)$ so that $$ \overline{V_{m + 1}}^{k_m}  \cup g_m\overline{V_{m + 1}}^{k_m}  \subseteq V_m.$$  Select $g_{m+1} \in V_{m + 1} \cap X$ with $O(\phi(g_{m + 1}^{k_m \cdots k_0})) > m + 2$, pick $j_{m + 1} > j_m$ with $O(\pi_{j_{m + 1}} \circ \phi(g_{m + 1}^{k_m \cdots k_0})) > m + 2$ and let $k_{m + 1} = N_{j_{m + 1}}!$.  Note that for any sequence $\overline{\epsilon}$ of zeros and ones we have $$V \supseteq g_0^{\epsilon_0}\overline{V_0}^{k_0} \supseteq g_0^{\epsilon_0}(g_1^{\epsilon_1}\overline{V_1}^{k_1})^{k_0} \supseteq \cdots$$ and so we may select an element $g_{0, \overline{\epsilon}}$ in the (necessarily nonempty) intersection of this decreasing chain.  Evidently $g_{0, \overline{\epsilon}} \in V$, and arguing as in the completely metrizable case we have in fact that $g_{0, \overline{\epsilon}} \in V \cap G$.  Also, for an arbitrary $n \in \mathbb{N}$ there exists an element $g \in \overline{V_{n + 1}}$ such that $$g_{0, \overline{\epsilon}} = g_0^{\epsilon_0}(g_1^{\epsilon_1}(g_2^{\epsilon_2}(\cdots g_n^{\epsilon_n}(g)^{k_n} \cdots)^{k_2})^{k_1})^{k_0}$$ and therefore $$\pi_{j_n} \circ \phi(g_{0, \overline{\epsilon}}) = \pi_{j_n} \circ \phi(g_0^{\epsilon_0}(g_1^{\epsilon_1})^{k_0}(g_2^{\epsilon_2})^{k_1k_0} \cdots (g_n^{\epsilon_n})^{k_{n  - 1} \cdots k_0}).$$  Now one can obtain conclusions $(*)_1$, $(*)_2$, and $(*)_3$ precisely as before.

To see why $(*)_4$ holds we once again let $h_m = \phi(g_m^{k_{m-1}\cdots k_0})$ and consider cases.  If $h_m$ has infinite order, then the limit (in the parameter $n \geq m$) of $O(\pi_{j_n}(h_m))$ is $\infty$.  Letting $\overline{\epsilon}$ be $1$ precisely at $m$, for each $n \geq m$ we have $\pi_{j_n} \circ \phi(g_{0, \overline{\epsilon}}) = \pi_{j_n}(h_m)$ and therefore $\phi(g_{0, \overline{\epsilon}})$ has infinite order.  If $O(h_m)$ is finite for all $m \in \mathbb{N}$ then pick the increasing sequence $r_0 < r_1 < \cdots$ precisely as before.  The same argument yields that $O(h_{r_0} \cdots h_{r_{\ell + 1}}) \geq \ell + 3$.  Define $\overline{\epsilon}$ as before and once again $$O(\phi(g_{0, \overline{\epsilon}})) \geq O(\pi_{j_{r_{\ell + 2} - 1}} \circ \phi(g_{0, \overline{\epsilon}})) = O(h_{r_0} \cdots h_{r_{\ell + 1}}) \geq \ell + 3$$ for arbitrary $\ell \in\mathbb{N}$, so we are done.

\end{proof}

\begin{remark}
Note that dropping the requirement regarding the existence of a number $N_j \in \mathbb{N}$ for which $O(\pi_j\circ\phi(g)) \leq N_j$ allows the conclusion of Theorem \ref{main} to fail.  To see this, consider the homomorphic retraction $\phi: S^1 \rightarrow \mathbb{Q}/\Z$ of abelian groups.  Give $\mathbb{Q}/\Z$ the trivial inverse limit structure under which $\pi_{i+1, i}: \mathbb{Q}/\Z \rightarrow \mathbb{Q}/\Z$ is the identity isomorphism.  The domain of $\phi$ is the compact, abelian, completely metrizable topological group $S^1$.  The set $X = \mathbb{Q}/\Z \subseteq S^1$ has $\phi(X)$ generating an abelian group and for each $U \in \Nb(S^1, 1)$ we have $\sup\{O(\phi(g)) \mid g \in U \cap X\} = \infty$.  Clearly no $\overline{X}_V$ can ever be found for any $V \in \Nb(S^1, 1)$ as in the conclusion, since the codomain of $\phi$ is countable and torsion.  Of course, the infinite divisible group $\mathbb{Q}/\Z$ is not residually finite.
\end{remark}

We will now show how to use Theorem \ref{main} to prove slenderness properties.

\begin{definition}\label{defn: residually bounded torsion}
    Recall that a group $K$ is \emph{bounded torsion} if there exists $n \in \mathbb{N}$ such that $g^n = 1$ for all $g \in K$.  We will say a group $H$ is a \emph{residually (bounded torsion)} if for all $g \in H \setminus \{1\}$ there exists a normal subgroup $L \unlhd H$ with $g \notin L$ and $H/L$ bounded torsion. 
\end{definition}  

It is straightforward to see that if $L_1$ and $L_2$ are normal in $H$ and $H/L_j$ is bounded torsion for each $j$, then $H/(L_1 \cap L_2)$ is bounded torsion.  The following is \cite[Proposition 3.5]{CoCo}.

\begin{lemma}\label{fromCoCo}  If $\mathcal{G}$ is either completely metrizable or locally compact Hausdorff, $\phi: \mathcal{G} \rightarrow H$ is an abstract group homomorphism with range of size less than $2^{\aleph_0}$ then there exists $V \in \Nb(\mathcal{G}, 1)$ for which any $V \supseteq U \in \Nb(\mathcal{G}, 1)$ has $\phi(U) = \phi(V)$.
\end{lemma}

We point out that the analogue of Lemma \ref{fromCoCo} also holds for locally countably compact Hausdorff topological groups, using a completely analogous argument.  We state it without proof.

\begin{lemma}\label{Artinianbutforcountably}  If $\mathcal{G}$ is locally countably compact Hausdorff, $\phi: \mathcal{G} \rightarrow H$ is an abstract group homomorphism with range of size less than $2^{\aleph_0}$ then there exists $V \in \Nb(\mathcal{G}, 1)$ for which any $V \supseteq U \in \Nb(\mathcal{G}, 1)$ has $\phi(U) = \phi(V)$.
\end{lemma}

\begin{theorem}\label{rebtorsion}  Suppose $\phi: \mathcal{G} \rightarrow H$ is an abstract group homomorphism with

\begin{itemize}
\item $\mathcal{G}$ completely metrizable or locally countably compact Hausdorff; and

\item $H$ residually (bounded torsion) with $|H| < 2^{\aleph_0}$.
\end{itemize}

\noindent Then there exists $V \in \Nb(\mathcal{G}, 1)$ for which

\begin{enumerate}

\item[$(\dagger)_1$]  $\phi(V)$ is a subgroup of $H$; and

\item[$(\dagger)_2$] every abelian subgroup of $\phi(V)$ is bounded torsion

\end{enumerate}

\noindent and so in particular $\phi(V)$ is a torsion group.

\end{theorem}

\begin{proof}
Assume the hypotheses.  Let $V \in \Nb(\mathcal{G}, 1)$ be as in the conclusion of Lemma \ref{fromCoCo} (or of Lemma \ref{Artinianbutforcountably}).  It is easy to check that $\phi(V)$ is a subgroup.  Suppose for contradiction that there exists a sequence of pairwise commuting elements $a_0, a_1, \ldots$ in $\phi(V)$ for which $O(a_i) \geq i + 2$.  Pick $L_i \unlhd H$ for which $a_i, a_i^2, \ldots, a_i^{i + 1} \notin L_i$ and $H/L_i$ is bounded torsion.  Define $\overline{L_i} = L_0 \cap L_1 \cap \cdots \cap L_i$, $H_i = H/\overline{L_i}$ and let $\pi_{i + 1, i}: H_{i + 1} \rightarrow H_i$ be the natural map.  Let $\pi: H \rightarrow \varprojlim H_i$ be the natural homomorphism,   choose $N_i$ such that $O(g)\leq N_i$ for all  $g\in H_i$, $X = \phi^{-1}(\{a_i\}_{i \in \mathbb{N}})$.  Now the homomorphism $\pi \circ \phi: \mathcal{G} \rightarrow \varprojlim H_i$ contradicts the conclusion of Theorem \ref{main} (since $|\pi \circ \phi(V)| \leq |\phi(V)| \leq |H| < 2^{\aleph_0}$).

\end{proof}

We note that some of the hypotheses and conclusions of Theorem \ref{rebtorsion} have similar flavor to the main result of \cite{KMV} (they have $\mathcal{G}$ being locally compact Hausdorff, in which setting one has more generous structure theorems).

\begin{definition}
    A group $H$ is \emph{cm-slender} (respectively \emph{lccH-slender}) if every abstract group homomorphism from a completely metrizable (respectively locally countably compact Hausdorff) topological group to $H$ has open kernel.  
\end{definition}

Clearly a locally compact Hausdorff group is locally countably compact Hausdorff (and interestingly the reverse implication does not hold \cite{HvRS}).  Thus lccH-slender implies lcH-slender.  The following is immediate.

\begin{corollary}\label{torsionfree cm-slender}  If $H$ is a torsion-free, residually (bounded torsion) group with $|H| < 2^{\aleph_0}$ then $H$ is cm-slender and lccH-slender.  In particular the conclusion holds for $H$ a torsion-free, residually finite group of size less than $2^{\aleph_0}$.
\end{corollary}

We point out that Theorem \ref{main} has natural companion results when the topological group $\mathcal{G}$ is replaced with the fundamental group $\mathbb{H}$ of the infinite earring.  Recall that $\mathbb{H}$ has a combinatorial description which essentially renders $\mathbb{H}$ as a group of words (of up to countable length) using letters in a countable alphabet $\{b_m^{\pm 1}\}_{m = 0}^{\infty}$ where each letter gets used only finitely often.  For each $s \in \mathbb{N}$ we have a subgroup $\mathbb{H}^s$ consisting of those words using only letters in $\{b_m^{\pm 1}\}_{m = s}^{\infty}$.  See \cite{CaCo1} for more details and background.

\begin{theorem}\label{mainforH}  Suppose that

\begin{itemize}

\item $G$ is a finite index subgroup of $\mathbb{H}$

\item $\{\pi_{i + 1, i}: H_{i + 1} \rightarrow H_i \mid i \in \mathbb{N}\}$ is a collection of group homomorphisms;

\item $\phi: G \rightarrow \varprojlim H_i$ is an abstract group homomorphism;

\item for each $j \in \mathbb{N}$ there is a natural number $N_j \in \mathbb{N}$ such that for all $g \in G$, $O(\pi_j \circ \phi(g)) \leq N_j$;

\item there is a subset $X \subseteq G$ such that $\phi(X)$ generates an abelian group and for each $s \in \mathbb{N}$ we have $\sup\{O(\phi(g)) \mid g \in \mathbb{H}^s \cap X\} = \infty$.

\end{itemize}

\noindent Then for each $s \in \mathbb{N}$ there is a subset $\overline{X}_s \subseteq \mathbb{H}^s \cap G$ such that

\begin{enumerate}

\item[$(*)_1$]  $|\overline{X}_s| = 2^{\aleph_0}$;

\item[$(*)_2$] $\phi \upharpoonright \overline{X}_s$ is injective;

\item[$(*)_3$] $\phi(\overline{X}_s)$ generates an abelian group;

\item[$(*)_4$] $\phi(\overline{X}_s)$ contains an element of infinite order.
\end{enumerate}

\end{theorem}

The proof is a simplified version of that of Theorem \ref{main}.  Since an appropriate analogue of Lemma \ref{fromCoCo} holds for $\mathbb{H}$ (see \cite[Theorem 4.4]{CaCo2} and its surrounding discussion) we also have the following two consequences to Theorem \ref{mainforH}.

\begin{theorem}\label{rebtorsionforH}  Suppose $\phi: \mathbb{H} \rightarrow H$ is an abstract group homomorphism with $H$ residually (bounded torsion) with $|H| < 2^{\aleph_0}$.  Then there exists $s \in \mathbb{N}$ for which every abelian subgroup of $\phi(\mathbb{H}^s) \leq H$ is bounded torsion (so in particular $\phi(\mathbb{H}^s)$ is a torsion group).
\end{theorem}

\begin{definition}
    A group $H$ is \emph{n-slender} if for every abstract homomorphism $\phi: \mathbb{H} \rightarrow H$ there exists $s \in \mathbb{N}$ such that $\ker(\phi) \geq \mathbb{H}^s$.
\end{definition} 

\begin{corollary}\label{torsionfree n-slender}  If $H$ is a torsion-free, residually (bounded torsion) group with $|H| < 2^{\aleph_0}$, then $H$ is n-slender.  In particular the conclusion holds for $H$ a torsion-free, residually finite group of size less than $2^{\aleph_0}$.
\end{corollary}

\end{section}

\begin{section}{Applications to ultraproduct images}\label{sec: applications to ui}

We discuss homomorphic images of ultraproducts.  A good reference for the set-theoretic concepts is \cite{Jech}.

\begin{definition}\label{ultrafilter}  Recall that an \emph{ultrafilter $\mathcal{U}$ on a set $X$} is a collection of subsets of $X$ such that

\begin{enumerate}

\item $\emptyset \notin \mathcal{U}$;

\item for each $Y \subseteq X$ either $Y \in \mathcal{U}$ or $X \setminus Y \in \mathcal{U}$;

\item $X_0, X_1 \in \mathcal{U}$ implies $X_0 \cap X_1 \in \mathcal{U}$

\end{enumerate}

\noindent and the ultrafilter $\mathcal{U}$ is

\begin{enumerate}

\item[(4)] \emph{nonprincipal} if $\{x\} \notin \mathcal{U}$ for each $x \in X$.

\end{enumerate}
\end{definition}

Taking an ultraproduct of a collection of structures is a standard mathematical tool, which we will briefly recall.  If $\mathcal{U}$ is an ultrafilter on a set $I$ and $\{G_i\}_{i \in I}$ is a collection of groups, then the \emph{ultraproduct} $\prod_{i \in I} G_i/\mathcal{U}$ is the quotient of $\prod_{i \in I} G_i$ by the normal subgroup $\{(g_i)_{i \in I} \mid \{i \in I \mid g_i = 1\} \in \mathcal{U}\}$.  We'll begin with the following.

\begin{theorem}\label{withultraprod}  If $\phi: \prod_{n \in \mathbb{N}} G_n/\mathcal{U} \rightarrow H$ is an abstract group homomorphism with

\begin{itemize}

\item $\mathcal{U}$ a nonprincipal ultrafilter; and

\item $H$ a residually (bounded torsion) group with $|H| < 2^{\aleph_0}$

\end{itemize}

\noindent then every abelian subgroup of the image of $\phi$ is bounded torsion.
\end{theorem}

\begin{proof}
The group $\prod_{n \in \mathbb{N}} G_n$ can be naturally topologized by considering each $G_n$ to be discrete and inducing the product topology.  It is well-known that this topological group is completely metrizable and has a basis of neighborhoods of identity given by subgroups of the form $\{1_{G_0}\} \times \{1_{G_1}\} \times \cdots \times \{1_{G_M}\} \times \prod_{n > M} G_n$.  Letting $\pi:  \prod_{n \in \mathbb{N}} G_n \rightarrow \prod_{n \in \mathbb{N}} G_n/\mathcal{U}$ be the quotient homomorphism, we know by the nonprincipality of $\mathcal{U}$ that the image under $\pi$ of any neighborhood of identity is $\prod_{n \in \mathbb{N}} G_n/\mathcal{U}$.  Now considering $\phi \circ \pi$ the claim follows from Theorem \ref{rebtorsion}.
\end{proof}

One fundamental question asks whether a group is the homomorphic image of a nonprincipal ultraproduct of groups (see \cite{Bergman} and \cite{Co}).  

\begin{definition}
    We say a group $H$ is \emph{ui} if there exists a sequence of groups $\{G_n\}_{n \in \mathbb{N}}$, a nonprincipal ultrafilter $\mathcal{U}$ on $\mathbb{N}$ and a surjective homomorphism from $\prod_{n \in \mathbb{N}} G_n/\mathcal{U}$ onto $H$ \cite[\textsection 3]{Co}.  
\end{definition}

\begin{corollary} \label{notui}  Suppose $H$ is residually (bounded torsion) and has an abelian subgroup which is not bounded torsion.  If either

\begin{itemize}

\item $H$ is torsion; or

\item $|H| < 2^{\aleph_0}$

\end{itemize}

\noindent then $H$ is not ui.

\end{corollary}

\begin{proof}
In case $|H| < 2^{\aleph_0}$ we can apply Theorem \ref{withultraprod}.  In case $H$ is torsion we suppose $H$ is a homomorphic image of an ultraproduct $\prod_{n \in \mathbb{N}} G_n/\mathcal{U}$ with $\mathcal{U}$ nonprincipal.  Then completely topologizing $\prod_{n \in \mathbb{N}} G_n$ in the natural way we apply Theorem \ref{main} $(*)_4$ to obtain a contradiction.

\end{proof}

We remark that the concept of ui group has close ties to that of measurable cardinals.   We say an ultrafilter $\mathcal{U}$ is \emph{$\kappa$-complete} if whenever $\mathcal{J} \subseteq \mathcal{U}$ has $|\mathcal{J}| < \kappa$ we have $\bigcap \mathcal{J} \in \mathcal{U}$.  Of course, ultrafilters are generally $\aleph_0$-complete (by Definition \ref{ultrafilter} (3)) and so the concept of $\kappa$-completeness is a strengthening of the fact that ultrafilters are closed under finite intersections.  A cardinal $\kappa > \aleph_0$ is \emph{measurable} if there exists a nonprincipal ultrafilter $\mathcal{U}$ on a set $X$, with $|X| = \kappa$, such that $\mathcal{U}$ is $\kappa$-complete.

A measurable cardinal must be quite large (it must be a so-called Mahlo cardinal \cite[Lemma 10.21]{Jech} and hence inaccessible).  As such, it is consistent with the standard ZFC axioms of set theory that there are no measurable cardinals.  It is a standard fact \cite[Exercise 6G.8]{Moschovakis} that if $X$ has a nonprincipal $\aleph_1$-complete ultrafilter, then $|X|$ is at least as large as the smallest measurable cardinal (so in particular, a measurable cardinal exists).  Thus there cannot exist an $\aleph_1$-complete nonprincipal ultrafilter on, say, $\mathbb{N}$.  Although measurable cardinals may seem like some abstruse idea only applicable in formal set theory, they are a common fixture in abelian group theory \cite{Eda}, non-abelian group theory \cite{EdaShelah} and even the theory of fundamental groups \cite{KeeslingRudyak}, \cite{Prz}.

\begin{lemma}\label{tocountable} \cite[Lemma 10]{Bergman}  The following are equivalent.

\begin{enumerate}

\item[(a)]  $H$ is a ui group.

\item[(b)]  There exists a collection $\{G_i\}_{i \in I}$ of groups, a nonprincipal ultrafilter $\mathcal{V}$ on $I$ which is not  $\aleph_1$-complete, and a surjective homomorphism from $\prod_{i \in I} G_i/\mathcal{V}$ to $H$.
\end{enumerate}

\end{lemma}

%\begin{proof}
%That (a) implies (b) is easy since there cannot exist a nonprincipal ultrafilter on $\mathbb{N}$ which is $\aleph_1$-complete.  Assume (b).  Since $\mathcal{V}$ is not countably complete, let $\{X_j\}_{j \in \mathbb{N}} \subseteq \mathcal{V}$ with $\bigcap_{j \in \mathbb{N}} X_j \notin \mathcal{V}$.  By taking finite intersections we may assume that $X_0 \supseteq X_1 \supseteq \cdots$, by skipping in the sequence we can also assume that $X_j$ includes $X_{j + 1}$ properly, and by subtracting the intersection $\bigcap_{j \in \mathbb{N}} X_j$ from all the terms in the descending chain we may further assume that $\bigcap_{j \in \mathbb{N}} X_j = \emptyset$.  By replacing $X_0$ with $I$ we may further assume that $X_0 = I$.  Letting $Y_j = X_j \setminus X_{j + 1}$ we see that $Y_j \notin \mathcal{V}$, the $Y_j$ are pairwise disjoint and nonempty, and that $\bigcup_{j \in \mathbb{N}} Y_j = I$.  Let $K_j = \prod_{i \in Y_j} G_i$ and $\mathcal{U} = \{X \subseteq \mathbb{N} \mid \bigcup_{j \in X}Y_j \in \mathcal{V}\}$.  Now combining the surjective homomorphisms $$\prod_{j \in \mathbb{N}}K_j/\mathcal{U} \rightarrow \prod_{i \in I}G_i/\mathcal{V} \rightarrow H$$ allows us to conclude (a).

%\end{proof}

\begin{lemma}\label{equivalencesmallerthanmeas}  Suppose a group $H$ is smaller than the smallest measurable cardinal (provided such a cardinal exists), $H$ is not a ui, and that $I$ is a set.  The following are equivalent.

\begin{enumerate}[(i)]

\item $H$ is a homomorphic image of a nonprincipal ultraproduct $\prod_{i \in I}G_i/\mathcal{U}$.

\item  $|I|$ is larger than the least measurable cardinal.

\end{enumerate}
\end{lemma}

\begin{proof}
Supposing (i) we have by Lemma \ref{tocountable} that $\mathcal{U}$ is $\aleph_1$-complete and so (ii) holds.  Supposing (ii) on the other hand, take $\kappa$ to be the smallest measurable cardinal.  Then $|I| \geq \kappa$ and there exists an ultrafilter $\mathcal{V}$ on $\kappa$ for which $\mathcal{J} \subseteq \mathcal{V}$ with $|\mathcal{J}| < \kappa$ implies $\bigcap \mathcal{J} \in \mathcal{V}$.  Without loss of generality $\kappa \subseteq I$ and we define $\mathcal{U} = \{X \subseteq I \mid X \cap \kappa \in \mathcal{V}\}$.  Note that $\mathcal{U}$ is a nonprincipal ultrafilter.  Letting $G_i = H$ for all $i \in I$ we obtain a surjective homomorphism from $\prod_{i \in I} G_i/\mathcal{U}$ to $H$ by letting $(h_i)_{i \in I}/\mathcal{U} \mapsto h$ where $\{i \in I \mid h_i = h\} \in \mathcal{U}$ (indeed the homomorphism is an isomorphism).  Thus (i) holds.

\end{proof}

\end{section}

\begin{section}{Applications to automorphism groups of rooted trees}\label{sec: applications to rooted trees}

Suppose that $\mathcal{X}$ is a finite nonempty set.  Let $\mathcal{X}^n$ denote the set of finite sequences $(x_i)_{0 \leq i <n} = (x_i)_n$ of elements of $\mathcal X$ having length $n$.  Let $\mathcal{X}^* = \bigcup_{n = 0}^{\infty} \mathcal{X}^n$ and for $s, s' \in \mathcal{X}^*$ we will say that $s = (x_i)_m$ is a \emph{prefix} of $s' = (x_i')_{n}$ (writing $s \preceq s'$) if $m \leq n$ and $x_i = x_i'$ for all $i < m$.

We will be considering subgroups of the automorphism group $\Aut(\mathcal{X}^*, \preceq)$ on the tree $(\mathcal{X}^*, \preceq)$ (with action written on the left).  If $H \leq \Aut(\mathcal{X}^*, \preceq)$ and $s \in \mathcal{X}^*$ we denote $H[s] = \{h \in H \mid hs' = s' \text{ for all } s \not\preceq s'\}$.  So, $H[s]$ is the subgroup of $H$ consisting of precisely those elements whose support is included in the subtree $\{t \in \mathcal{X}^* \mid s \preceq t\}$.  Given $h \in H$ and $s \in \mathcal X^*$ we write $\Orb(h, s)$ for the orbit of $s$ under the group action of $\langle h \rangle \leq H$.

\begin{lemma}\label{largeorders}  If $H \leq \Aut(\mathcal{X}^*, \preceq)$ is such that $H[s]$ is nontrivial for each $s \in \mathcal{X}^*$, then each $H[s]$ has elements of arbitrarily high (possibly infinite) order.
\end{lemma}

\begin{proof}
Assume the hypothesis and fix $s \in \mathcal{X}^*$.  For $h \in H[s]$ we define a binary relation $R_h$ on elements of $\mathcal{X}^*$ by letting $t R_h t'$ if $t \prec t'$ and there exists $t \prec t''$ with $t' \neq t''$ and $\Orb(h, t') = \Orb(h, t'')$.  Clearly $t R_h t'$ implies that $|\Orb(h, t')| \geq 2|\Orb(h, t)|$, and if $t \in \mathcal{X}^m$ and $h_0 \in H[t]$ then $\Orb(h, t) = \Orb(h_0h, t)$.

Note that if there does not exist $t'$ such that $t R_h t'$ then $|\Orb(h, t)| = |\Orb(h, t')|$ for each $t \prec t'$ and in fact for every $k \in \mathbb{N}$ we know $h^kt'$ is the unique element of $\Orb(h, t')$ which extends $h^kt$.  In this case we may (by assumption) select a nontrivial $h_0 \in H[t]$ and $t \prec t'$ with $h_0t' \neq t'$, and now $\Orb(h, t) = \Orb(h_0h, t)$ and $t R_{h_0h} t'$ (the latter holds since $(h_0h)^kt' = h^kt'$ for $0 \leq k < |\Orb(h, t')|$ and $(h_0h)^{|\Orb(h, t')|}t' = h_0t'$).

Now we are ready to prove the lemma.  We shall produce a strictly ascending chain $t_0 \prec t_1 \prec \cdots$ together with a sequence $h_0, h_1, \ldots$ of elements of $H[s]$ by induction.  Let $t_0 = s$ and take $h_0 \in H[t_0]$ to be nontrivial.  As $h_0$ fixes $t_0$ we may select $t_0 \prec t_1$ with $t_0 R_{h_0} t_1$.  If it is possible to select $t_2$ with $t_1 R_{h_0} t_2$ then do so and set $h_1$ to be identity.  If it is not possible to select such a $t_2$ then select nontrivial $h_1 \in H[t_1]$ and $t_2$ for which $h_1t_2 \neq t_2$ and we have $t_1 R_{h_1h_0} t_2$.  Continuing in this way we have $$t_0 R_{h_0} t_1 R_{h_1h_0} t_2 R_{h_2h_1h_0} t_3 \cdots$$ with $h_i \in H[t_i]$.  Therefore $|\Orb(h_{\ell} \cdots h_0, t_{\ell + 1})| \geq 2^{\ell + 1}$ and so the order of $h_{\ell} \cdots h_0$ is at least $2^{\ell + 1}$ and we are done.
\end{proof}

\begin{lemma}\label{applicabilitytotreeation}   If $H \leq \Aut(\mathcal{X}^*, \prec)$ has $H[s]$ is nontrivial for each $s \in \mathcal{X}^*$ then there exists a set of elements $\{a_0, a_1, \ldots \} \subseteq H$ which generate an abelian subgroup and $O(a_m) \geq m + 2$.
\end{lemma}

\begin{proof}
Let $\{s_0, s_1, s_2, \ldots\}$ be a collection of elements in $\mathcal{X}^*$ such that $s_m \not \preceq s_n$ if $m \neq n$.  (It is easy to find such a collection since by assumption $H = H[\emptyset]$ is nontrivial and therefore $\mathcal{X}$ has at least two elements.)  By Lemma \ref{largeorders} find $a_m \in H[s_m]$ with $O(a_m) \geq m + 2$.  That $\{a_0, a_1, \ldots \}$ generates an abelian subgroup is clear.
\end{proof}

\begin{theorem}\label{thmfortreegroups}  Suppose $H \leq \Aut(\mathcal{X}^*, \prec)$ and $H[s]$ is nontrivial for each $s \in \mathcal{X}^*$.

\begin{enumerate}

\item If $|H| < 2^{\aleph_0}$ and $\phi: G \rightarrow H$ is a homomorphism with $G$ completely metrizable or locally countably compact then there is some $V \in \Nb(G, 1)$ for which $\phi(V) \neq H$.

\item If $|H| < 2^{\aleph_0}$ and $\phi: \mathbb{H} \rightarrow H$ is a homomorphism then there is some $N \in \mathbb{N}$ for which $\phi(\mathbb{H}^N) \neq H$.

\item If $H$ is torsion or $|H| < 2^{\aleph_0}$ then $H$ is not a ui-group.

\end{enumerate}
\end{theorem}

\begin{proof}
Use Lemma \ref{applicabilitytotreeation} together with Theorem \ref{rebtorsion} for (1), with Theorem \ref{rebtorsionforH} for (2), with Corollary \ref{notui} for (3).
\end{proof}

Note that the group $H = \Aut(\mathcal{X}^*, \prec)$ is a compact Polish group under the natural topology.  As such, it is easily seen to be a ui group (an ultrapower of $H$ maps onto $H$), so it is clear why the assumptions in Theorem \ref{thmfortreegroups} (3) are essential.  Moreover the identity homomorphism $H \rightarrow H$ witnesses why the assumption $|H| < 2^{\aleph_0}$ is essential in Theorem \ref{rebtorsion}.

Recall that a group $H \leq \Aut(\mathcal{X}^*, \prec)$ is \emph{weakly branch} if it satisfies the hypotheses of Lemma \ref{largeorders} and has level transitive action on $\mathcal{X}^*$ \cite[Definition 1.13]{BartholdiGrigorchukSunik}.  Examples of weakly branch groups include the Grigorchuk group \cite{Grigorchuk} and the Basilica group \cite{GrigorchukZuk}, both of which are finitely generated.  So, Grigorchuk's group is an example of a finitely generated infinite torsion group which is not ui.  Since the Basilica group is torsion-free, countable, and residually finite we know by Corollaries \ref{torsionfree cm-slender} and \ref{torsionfree n-slender} that it is slender in each of the non-abelian senses considered in this paper.  By combining Theorem \ref{thmfortreegroups} (3) and Lemma \ref{equivalencesmallerthanmeas} we obtain the following.

\begin{corollary}\label{Ggroupmeasurable} Grigorchuk's group (or indeed any countable non-(ui) group) is a homomorphic image of a nonprincipal ultraproduct of groups if and only if there exists a measurable cardinal.
\end{corollary}

\end{section}

\end{document}